\documentclass{amsart}
\usepackage{mathptmx, bbm, amscd, amssymb, enumerate, colonequals, mathdots, xcolor, comment, graphicx, lineno, xspace, amsmath, mathtools}
\usepackage{color}
\usepackage{tikz-cd}
\definecolor{chianti}{rgb}{0.6,0,0}
\definecolor{meretale}{rgb}{0,0,.6}
\definecolor{leaf}{rgb}{0,.35,0}
\usepackage[colorlinks=true, hyperindex, citecolor=meretale, urlcolor=leaf, linkcolor=chianti]{hyperref}
\usepackage[all]{xy}
\usepackage{paralist}
\usepackage{mathrsfs}
\usepackage{parskip}
\usepackage{etoolbox}
\AtBeginEnvironment{proof}{\vspace*{-1.5em}}
\usepackage{enumitem}
\usepackage{ytableau}

\newtheorem{theorem}{Theorem}[section]
\newtheorem{lemma}[theorem]{Lemma}
\newtheorem{corollary}[theorem]{Corollary}
\newtheorem{prop}[theorem]{Proposition}
\newtheorem*{maintheorem}{Main Theorem}

\theoremstyle{definition}

\numberwithin{equation}{theorem}
\newtheorem{examplex}[theorem]{Example}
\newenvironment{example}{\pushQED{\qed}\examplex}{\popQED\endexamplex}
\newtheorem{remarkx}[theorem]{Remark}
\newenvironment{remark}{\pushQED{\qed}\remarkx}{\popQED\endremarkx}

\usepackage[backend=biber,style=alphabetic,doi=false,isbn=false,url=false,eprint=true,maxbibnames=10,minbibnames=3,mincitenames=3,maxcitenames=4,maxalphanames=4,minalphanames=3,backref=true]{biblatex}
\DeclareLabelalphaTemplate{%
  \labelelement{%
    \field[final]{shorthand}
    \field{label}
    \field[strwidth=3,strside=left,ifnames=1]{labelname}
    \field[strwidth=1,strside=left]{labelname}
  }
}
\DefineBibliographyStrings{english}{%
  backrefpage={},
  backrefpages={}
}

\usepackage{xpatch}
\DeclareFieldFormat{backrefparens}{\addperiod#1}
\DeclareFieldFormat{postnote}{#1}   
\xpatchbibmacro{pageref}{parens}{backrefparens}{}{}

\DeclareMathOperator{\ara}{ara}
\DeclareMathOperator{\cp}{cptd}
\DeclareMathOperator{\ecd}{\acute{e}cd}
\DeclareMathOperator{\lcd}{lcd}
\DeclareMathOperator{\rad}{rad}
\DeclareMathOperator{\init}{in}
\DeclareMathOperator{\trdeg}{trdeg}
\DeclareMathOperator{\Frac}{Frac}
\DeclareMathOperator{\chr}{char}

\DeclareMathOperator{\rank}{rank}
\DeclareMathOperator{\htt}{ht}
\DeclareMathOperator{\tr}{tr}
\DeclareMathOperator{\Gr}{Gr}

\DeclareMathOperator{\GL}{GL}
\DeclareMathOperator{\SL}{SL}
\DeclareMathOperator{\Sp}{Sp}
\DeclareMathOperator{\OO}{O}  
\DeclareMathOperator{\HH}{H}
\newcommand{\Hc}{\HH_{\text{c}}}

\newcommand{\fraka}{\mathfrak{a}}
\newcommand{\frakb}{\mathfrak{b}}
\newcommand{\frakm}{\mathfrak{m}}
\newcommand{\frakp}{\mathfrak{p}}
\newcommand{\frakU}{\mathfrak{U}}

\renewcommand{\AA}{\mathbb{A}}
\newcommand{\CC}{\mathbb{C}}
\newcommand{\NN}{\mathbb{N}}
\newcommand{\VV}{\mathbb{V}}
\newcommand{\ZZ}{\mathbb{Z}}

\renewcommand{\ge}{\geqslant}
\renewcommand{\le}{\leqslant}
\renewcommand{\subset}{\subseteq}

\renewcommand{\to}{\longrightarrow}
\renewcommand{\mapsto}{\longmapsto}
\newcommand{\into}{\lhook\joinrel\longrightarrow}
\newcommand{\minor}[2]{\IfBlankTF{#2}{[\; #1 \;]}{[\; #1 \mid #2 \;]}}
\newcommand{\md}[1]{{\left\lvert #1 \right\lvert}}
\newcommand{\andd}{\quad\text{and}\quad}
\newcommand{\smatrix}[1]{\left[\begin{smallmatrix} #1 \end{smallmatrix}\right]}
\newcommand{\ph}{\phantom{2}}

\newcommand{\boldone}{\mathbbm{1}}
\newcommand{\et}{\text{\'et}}
\newcommand{\mapsfrom}{\mathrel{\reflectbox{\ensuremath{\mapsto}}}}

\newcommand{\etale}{\'etale}
\newcommand{\Poincare}{Poincar\'e}

\NewDocumentCommand{\JPSW}{o}{\IfNoValueTF{#1}{\cite{JPSW}}{\cite[#1]{JPSW}}}
\NewDocumentCommand{\BMMP}{o}{\IfNoValueTF{#1}{\cite{BMMP}}{\cite[#1]{BMMP}}}

\begin{document}
\title[Arithmetic rank of classical nullcones]{The arithmetic rank of \\ nullcones of classical invariant rings}

\author{Manav Batavia}
\address{Department of Mathematics, Purdue University, West Lafayette, IN~47907, USA}
\email{mbatavia@purdue.edu}

\author{Aryaman Maithani}
\address{Department of Mathematics, University of Utah, Salt Lake City, UT~84112, USA}
\email{maithani@math.utah.edu}

\author{Kesavan Mohana Sundaram}
\address{Department of Mathematics, University of Nebraska-Lincoln, NE~68588, USA}
\email{km2@huskers.unl.edu}

\begin{abstract}
  We compute the arithmetic rank of nullcone ideals arising from natural actions of the special linear, orthogonal, and symplectic groups on direct sums of copies of their standard and dual representations. Over an infinite field of characteristic different from two, we show that the arithmetic rank equals the Krull dimension of the invariant ring, and compute this dimension. In positive characteristic, the required lower bounds are not detected by local cohomology, and are obtained using \'etale cohomology. 
\end{abstract}
\maketitle

\section{Introduction}

  Let $K$ be a field, and consider a group $G$ acting on an $\NN$-graded $K$-algebra $S$ via a degree-preserving $K$-algebra automorphism. 
  The \emph{nullcone ideal} of this action is the ideal $\frakU \colonequals \frakm S$, where $\frakm$ is the homogeneous maximal ideal of the invariant ring $S^{G}$. 
  In other words, $\frakU$ is the ideal generated by homogeneous invariant polynomials of positive degree. 
  The \emph{arithmetic rank} of an ideal $\fraka \subseteq S$ is defined to be the least number of elements required to generate $\fraka$ up to radical, i.e.,
  \begin{equation*} 
    \ara \fraka \colonequals
    \min\{ k \ge 0 : \text{there exist $f_{1}, \cdots, f_{k} \in S$ such that $\rad(f_{1}, \ldots, f_{k})S = \rad \fraka$} \}.
  \end{equation*}

  In this paper, we compute the arithmetic rank of nullcone ideals arising in classical invariant theory. 
  Assume now that $K$ is an infinite field of characteristic other than two, and let $X$ and $Y$ be matrices of indeterminates of sizes $m \times t$ and $t \times n$ respectively. The group $\GL_{t}(K)$ acts on the polynomial ring $S \colonequals K[X_{m \times t}, Y_{t \times n}]$, 
  where $M \in \GL_{t}(K)$ acts via
  \begin{equation*} 
    M \colon 
    \begin{cases}
      X \mapsto X M^{-1}, \\
      Y \mapsto M Y.
    \end{cases}
  \end{equation*}
  In other words, we have $n$ copies of the regular representation of $\GL_{t}(K)$, along with $m$ copies of the dual representation. 
  In turn, this action restricts to an action of any subgroup $G$ of $\GL_{t}(K)$. 
  The four classical groups are 
  the general linear group $\GL_{t}(K)$,
  the special linear group $\SL_{t}(K)$, 
  the orthogonal group
  \begin{equation*} 
    \OO_{t}(K) \colonequals \{M \in \GL_{t}(K) : M^{\tr} M = \boldone_{t}\},
  \end{equation*}
  and
  the symplectic group
  \begin{equation*} 
    \Sp_{2t}(K) \coloneqq \{M \in \GL_{2t}(K) : M^{\tr} \Omega M = \Omega\},
  \end{equation*}
  where $\Omega \colonequals \smatrix{0 & \boldone_{t} \\ -\boldone_{t} & 0}$. 
  The rings of invariants were studied by Weyl \cite{WeylClassical} when $K = \CC$,
  and later by De Concini and Procesi \cite{DeConciniProcesi76} who extended the descriptions of the rings of invariants to infinite fields of positive characteristic.

    Special cases of these classical nullcone ideals and their arithmetic rank computations have appeared in a variety of guises in the literature. They can be summarised as follows. 
    \begin{enumerate}[leftmargin=*, align=left, label=(\alph*)]
        \item If $G=\GL_t(K)$, the nullcone ideal $\frakU$ is generated by the entries of $XY$. The minimal primes of $\frakU$ are the ideals
        \begin{equation} \label{eq:variety-complexes}
          \frakp_{i, j} \colonequals I_{i+1}(X) + I_{1}(XY) + I_{j+1}(Y)
        \end{equation}
        with $i+j=t$. These ideals define the varieties of complexes introduced in \cite{Buchsbaum-Eisenbud}. The properties of $\frakp_{i,j}$ and $\frakU$ have been studied extensively \cite{Kempf,HunekeVoC,ConciniStricklandVoC,MehtaTrivedi,PTW}. The arithmetic rank of the nullcone ideal $\frakU$ was computed in \cite[Theorem~4.1]{JPSW}.
        \item If $G=\SL_t(K)$ and $m=0$, i.e., $S=K[Y]$, then the nullcone ideal $\frakU$ is generated by the $t$-minors of $Y$, placing it within the vast classical theory of determinantal ideals \cite{Macaulay1916,HochsterEagon,BrunsVetter,Miro-Roig}. Its arithmetic rank was computed in \cite{BrunsSchwanzl}. 
        \item If $G=SL_t(K)$, $m=1$, and $n\geq t$, then the nullcone ideal $\frakU$ is $I_{1}(X Y) + I_{t}(Y)$.
        It arises naturally in linkage theory as the \emph{generic $n$-residual intersection ideal} of an ideal of variables \cite{HunekeUlrichGenericRI}.
        From a different perspective, $\frakU$ is the ideal $\frakp_{1,\;t-1}$ arising from a variety of complexes as in~\eqref{eq:variety-complexes}. Its arithmetic rank was computed in \cite[Theorem~4.7]{BMMP}.
        \item If $G=\OO_t(K)$ or $G=\Sp_{2t}(K)$, and $m=0$, the corresponding nullcone ideals have been investigated in \cite{KraftSchwarz,HJPS,PTW}. Their arithmetic ranks were computed in \cite[Theorems~5.1,~3.1]{JPSW}.
    \end{enumerate}
    
    We complete these computations by determining the arithmetic rank of the nullcone ideal in all the remaining cases.

    \begin{maintheorem}[Theorems~\ref{thm:ara-SL},~\ref{thm:ara-O}, and~\ref{thm:ara-Sp}]
      Let $K$ be an infinite field of characteristic not two, and
      let $G \subseteq \GL_{t}(K)$ be one of $\SL_{t}(K)$, $\OO_{t}(K)$, or $\Sp_{t}(K)$, where in the last case we assume that $t$ is even.
      Consider the aforementioned natural action of $G$ on the polynomial ring $S \colonequals K[X_{m \times t}, Y_{t \times n}]$, 
      where we assume that $t \le \max(m, n)$.
      The arithmetic rank of the nullcone ideal $\frakU \subset S$ satisfies
      \begin{equation*}
          \ara \frakU = \dim S^G.
      \end{equation*}
      Moreover, the above dimension is 
      \begin{equation*} 
         \dim S^G = \dim S - \dim G.
      \end{equation*}
    \end{maintheorem}
    We recall that the dimensions of the algebraic groups of interest are
    \begin{equation*}
      \dim \SL_t(K) = t^2-1, 
      \quad 
      \dim \OO_t(K) = \binom{t}{2}, 
      \quad
      \text{and}
      \quad
      \dim \Sp_t(K) = \binom{t+1}{2}.
    \end{equation*}

  We show that the actions of the orthogonal and symplectic groups on $K[X_{m \times t}, Y_{t \times n}]$ can be reduced to actions on $K[Y'_{t \times (m + n)}]$, i.e., only copies of the regular representation; 
  this reduction allows us to use the existing descriptions of the nullcone ideal and the arithmetic rank \cite{DeConciniProcesi76,JPSW}. 
  In view of this, most of the paper is devoted to studying 
  \begin{equation*}
      K[X,\, Y]^{\SL_{t}(K)} \;=\; K[X Y,\, \Delta_{t}(X),\, \Delta_{t}(Y)],
  \end{equation*} 
  where $\Delta_{t}(X)$ is the set of $t$-sized minors of $X$;
  the corresponding nullcone ideal is given as $I_{1}(X Y) + I_{t}(X) + I_{t}(Y)$. 

  When $K$ is a field of characteristic zero, the arithmetic rank of the nullcone ideals is readily seen to equal the Krull dimension $d\colonequals\dim(S^{G})$; we sketch an argument that may be found in \cite[p.~117]{Bruns:AdditionASL} or \JPSW[Theorem~1.1]. 
  The homogeneous maximal ideal $\frakm$ of $S^{G}$ can be generated up to radical by $d$ elements in $S^{G}$, a fortiori in $S$, giving us the upper bound. 
  On the other hand, the inclusion $S^{G} \into S$ splits in characteristic zero as $G$ is linearly reductive. 
  In turn, we obtain an injection
  \begin{equation} \label{eq:local-cohomology-injection}
    0 \neq \HH_{\frakm}^{d}(S^{G}) \into \HH_{\frakm S}^{d}(S)
  \end{equation}
  of local cohomology modules, giving us the lower bound; see \cite[Proposition~9.12]{24hours}. 
  However, in positive characteristic, it is known that the inclusion $S^{G} \into S$ is almost never split for classical invariant rings; see \cite{HJPS}. 
  In fact, we show that the local cohomology obstruction vanishes (Theorem~\ref{thm:small-lcd}), and the lower bound on arithmetic rank is instead obtained using \etale\ cohomology.

  In the spirit of the above argument, we first compute the Krull dimension of $S^{G}$, giving us the arithmetic rank in characteristic zero and an upper bound in positive characteristic. 
  We establish the nonvanishing of the critical \etale\ cohomology module to obtain the desired lower bound and show that the equality $\ara(\frakm S) = \dim(S^{G})$ holds even in odd positive characteristic. If $t>\min(m,n)$, we also obtain minimal set-theoretic generators of the nullcone ideal.
    
  The relevant facts and machinery regarding \etale\ cohomology are reviewed in Section~\ref{sec:etale}. In Section~\ref{sec:calculations}, we perform some preliminary calculations to be used in our main results. In Section~\ref{sec:SL}, we compute the arithmetic rank of the nullcone ideal for the $\SL$-action. Finally, Section~\ref{sec:OSp} establishes the Main Theorem for the orthogonal and symplectic group actions. Appendix~\ref{sec:Appendix} gives an elementary and self-contained proof of the lower bound on the dimension of the invariant ring used in Theorem~\ref{thm:dim-aux-ring}.

    \textbf{Acknowledgements.}
  We are grateful to Benjamin Briggs, Jack Jeffries, Vaibhav Pandey, Lisa Seccia, Anurag K.~Singh, and Uli Walther for several valuable discussions. M.B. was supported by NSF Grants DMS-2302430 and DMS-2100288, and by Simons Foundation Grant SFI-MPS-TSM-00012928. A.M. was supported by NSF grant DMS-2349623 and a Simons Dissertation Fellowship. K.M.S. was supported by NSF grants DMS-2236983, DMS-2601940, and DMS-2044833.

    \textbf{AI disclosure.}
    We used freely available ChatGPT models to assist with literature searches and with formalising certain ideas in the proof of Proposition~\ref{prop:cpd-Vr}. We wrote the manuscript in its entirety, and take full responsibility for its content and correctness.

\section{\'Etale cohomology} \label{sec:etale}

  In this section, we collect facts about \etale\ cohomology that we shall use in our computations later. 
  We work over an algebraically closed field $K$ of characteristic not two and 
  all varieties will be assumed to be quasiprojective.
  The following material may be found in \JPSW[\S6.3] or \cite{Milne}. 
  
  For a quasiprojective variety $V$ over $K$, we set
  \begin{align*} 
    \Hc^{i}(V) \colonequals \HH_{\text{c}, \et}^{i}(V; \ZZ/2) 
    \andd
    \HH^{i}(V) \colonequals \HH_{\et}^{i}(V; \ZZ/2),
  \end{align*}
  where the subscript $\text{c}$ denotes \etale\ cohomology with compact support. 
  The \emph{compact dimension} and \emph{\etale\ cohomological dimension} of $V$, respectively, are
  \begin{align*} 
    \cp(V) \colonequals \inf\{i \ge 0 : \Hc^{i}(V) \neq 0\} 
    \andd
    \ecd(V) \colonequals \sup\{i \ge 0 : \HH^{i}(V) \neq 0\}.
  \end{align*}
  The \emph{critical cohomology group} of $V$ is $\Hc^{\cp V}(V)$. 

  \emph{Long exact sequence of a subspace.} \cite[p.~94]{Milne}
  If $V = W \sqcup U$ with $U$ open in $V$, then there is a long exact sequence 
  \begin{equation} \label{eq:les-triple}
    \Hc^{i}(U) \to \Hc^{i}(V) \to \Hc^{i}(W) \to \Hc^{i+1}(U) \to.
  \end{equation}

  \emph{Mayer--Vietoris.} \cite[III.2.24]{Milne}
  If $V = V_{1} \cup V_{2}$ with $V_{i}$ open, then there is a long exact sequence
  \begin{equation} \label{eq:les-mv}
    \HH^{i}(V) \to
    \HH^{i}(V_{1}) \oplus \HH^{i}(V_{2}) \to
    \HH^{i}(V_{1} \cap V_{2}) \to
    \HH^{i+1}(V) \to.
  \end{equation}
  The same holds for \etale\ cohomology groups with compact supports as well, \JPSW[p.~26]. 

  \emph{Affine vanishing.} 
    \JPSW[p.~26]
    If $V$ is an affine variety of dimension $d$, then \cite[VI.7.2]{Milne} states
    \begin{equation*}
        \HH^{i}(V) = 0 \quad
        \text{for $i > d$}.
    \end{equation*}
  In turn, if $\frakU$ is an ideal in a polynomial ring $S$ 
  and $U \colonequals \VV(\frakU)^{c}$ the open complement of its variety, then repeated application of the Mayer--Vietoris sequence yields
  \begin{equation} \label{eq:affine-vanishing}
    \ecd(U) \le \ara(\frakU) + \dim(S) - 1,
  \end{equation}
    after noting that $\ara(\frakU)$ is the minimum number of affine open sets needed to cover $U$.
  This will be our key tool in obtaining lower bounds for arithmetic rank. 

  \emph{\Poincare\ duality.} \cite[VI.11.1]{Milne} 
  If $V$ is a smooth quasiprojective variety, then there is an isomorphism $\HH^{i}(V) \cong \Hc^{2 \dim(V) - i}(V)$ for every $i$. 
  In turn, we obtain the equality
  \begin{equation} \label{eq:PD}
    \ecd(V) = 2 \dim(V) - \cp(V).
  \end{equation}

  \begin{lemma}[{\JPSW[Lemma~6.1]}] \label{lem:cpdim-fiber}
    Let $F$, $E$, $B$ be quasiprojective varieties. 
    Assume that:
    \begin{enumerate}[label=(\arabic*), leftmargin=1cm]
      \item $F \to E \to B$ is a locally trivial fiber bundle;
      \item the critical cohomology group of the fiber $F$ has rank one;
      \item one of the following holds:
      \begin{enumerate}[label=(\alph*)]
        \item the base $B$ is simply connected with critical cohomology group of rank one, or
        \item $B$ is smooth of dimension $b$ as an algebraic variety, is covered by $k$ affines where $\cp B = b-k+1$, and has critical cohomology group of rank one.
      \end{enumerate}
    \end{enumerate}
    Then $\cp E = \cp F + \cp B$, and the critical cohomology group of $E$ has rank one. \qed
  \end{lemma}

\section{Preliminary calculations}
\label{sec:calculations}

    If $X$ is an $m \times n$ matrix over a ring $S$ and $t \ge 1$ an integer, we let $\Delta_{t}(X)$ denote the (possibly empty) set of $t$-sized minors. 
  In turn, $I_{t}(X)$ denotes the ideal of $S$ generated by $\Delta_{t}(X)$. 
  Given $i \in [m]$ and $j \in [n]$, we denote the $(i, j)^{\text{th}}$ entry of $X$ by $[X]_{i, j}$, or simply, $[X]_{i j}$.
  Given subsets $A \subseteq [m]$ and $B \subseteq [n]$ of size $t$, we denote by $\minor{A}{B}_{X}$ the $t$-sized minor of $X$ obtained by selecting rows and columns of $X$ whose indices lie in $A$ and $B$, respectively. 
  If $m = t$ (resp. $n = t$), we then omit $A$ (resp. $B$) from the notation with the understanding that we have $A = [t]$ (resp. $B = [t]$). 

  \begin{lemma} \label{lem:one-more-generator}
    If $\fraka \subset \frakb$ are ideals in a ring $S$ with $\frakb = \fraka + b S$,
    then $\ara \fraka \ge \ara \frakb - 1$.
  \end{lemma}
  \begin{proof} 
    If $\rad (\underline{a} S) = \rad \fraka$, then $\rad((\underline{a}, b)S) = \rad \frakb$, 
    giving us $\ara \frakb \le \ara \fraka + 1$. 
  \end{proof}

  \begin{lemma} \label{lem:ara-det-inside-det}
    Let $R \colonequals K[X_{m \times n}]/I_{r+1}(X)$,
    where we assume $r+1 \le n \le m$. 
    Consider the ideal $J \colonequals I_{r}(X) R$ inside the determinantal ring. 
    Then,
    \begin{equation*} 
      \ara J \le r(m+n-2r)+1.
    \end{equation*}
  \end{lemma}
  \begin{proof} 
    Recall that $R$ is an ASL on the poset $\Pi$ consisting of all $t$-sized minors of $X$ with $t \le r$; we refer the reader to \cite{Bruns:AdditionASL} for more details. 
    The set $\Omega \subset \Pi$ of $r$-sized minors forms a poset ideal that generates $J$.
    By \cite[Proposition~2.1]{Bruns:AdditionASL}, we obtain $\ara J \le \rank \Omega$, where $\rank \Omega$ is the longest length of a chain in $\Omega$. 
    The rank is at most $r(m+n-2r)+1$;
    indeed, the extremal elements of $\Omega$ are
    \begin{equation*} 
      \minor{1, \ldots, r}{1, \ldots, r}_{X} 
      \andd
      \minor{m-r+1, \cdots, m}{n-r+1, \ldots, n}_{X}, 
    \end{equation*} 
    and the sum of indices increases by at least one among consecutive elements in a chain. 
  \end{proof}

  \begin{lemma} \label{lem:cpd-locally-closed-union}
    Let $V$ be a variety with a partition $V = \bigsqcup_{i=0}^{t-1} V_{i}$, 
    and set $V_{\le j} \colonequals \bigsqcup_{i=0}^{j} V_{i}$ for $0 \le j < t$. 
    Assume that $V_{j}$ is open in $V_{\le j}$ (this implies that each $V_{j}$ is locally closed) and that $\cp V_{j+1} < \cp V_{j}$ for all $j$. 
    Then, $\cp V = \cp V_{t-1}$. 
  \end{lemma}
  \begin{proof}
    It suffices to show that $\cp V_{\le j} = \cp V_{j}$ for all $j$. 
    This is clear for $j = 0$, and for $j \ge 1$, it follows inductively from the long exact sequence~\eqref{eq:les-triple} applied to the decomposition $V_{\le j+1} = V_{j+1} \sqcup V_{\le j}$.
  \end{proof}

  \begin{prop} \label{prop:cpd-Vr}
    Let $t \ge 2$ be an integer, $K$ a field, and for $0 \le r < t$, set 
    \begin{equation*} 
      V_{r} \colonequals 
      \{
        (A, B) \in K^{t \times t} \times K^{t \times (t - 1)}
        :
        \rank A = r \text{ and } A B = 0
      \}.
    \end{equation*}
    Then, $\cp V_{r} = r^{2} + 2(t - r)(t - 1)$. 
  \end{prop}
  \begin{proof} 
    The statement is clear for $r=0$ as $V_{0} = \{0\} \times K^{t \times (t - 1)}$, 
    so we assume $1 \le r < t$ and set $V \colonequals V_{r}$ and
    $S \colonequals \{A \in K^{t \times t} : \rank A = r\}$. 

    By \cite[Lemma~6]{BrunsSchwanzl}, there is a locally trivial fiber bundle
    \begin{equation*} 
      \GL_{r} \to S \to \Gr(t-r, t) \times \Gr(r, t)
    \end{equation*}
    satisfying the hypotheses (1), (2), (3a) of Lemma~\ref{lem:cpdim-fiber}, and thus $\cp S = r^{2} + 0 = r^{2}$ with the critical cohomology group of $S$ being rank one; 
    we refer the reader to \cite[\S3~(4)]{BrunsSchwanzl} for the properties about grassmannians that we have used, namely that they are simply-connected, of compact dimension zero, and with critical cohomology group of rank one.

    We claim that there is a Zariski locally trivial fiber bundle
    \begin{equation} \label{eq:bundle}
      \begin{tikzcd}
        {K^{(t-r) \times (t-1)}} \arrow[r] &
        {V} \arrow[r, "\pi"] &
        {S,}
      \end{tikzcd}
    \end{equation}
    where $\pi$ maps $(A, B) \in V$ to $A$. 
    To this end, consider the open cover of $S$ given by the sets 
    \begin{equation*} 
      U_{\sigma, \tau} \colonequals \{A \in S : \minor{\sigma}{\tau}_{A} \neq 0\},
    \end{equation*}
    for $\sigma, \tau \subseteq [t]$ with $\md{\sigma} = \md{\tau} = r$.
    For each $U = U_{\sigma, \tau}$, we now show that 
    $\pi^{-1}(U) \to U$ is isomorphic to the projection
    $U \times K^{(t-r)(t-1)} \to U$. 
    Without loss of generality, we may assume that $\sigma = \tau = [r]$, so that $U = U_{[r], [r]}$. 
    Any $A \in U$ can thus be written as 
    \begin{equation*} 
      A = 
      \begin{bmatrix}
        A_{11} & A_{12} \\
        A_{21} & A_{22}
      \end{bmatrix}_{t \times t},
    \end{equation*}
    where $A_{11}$ is invertible of size $r \times r$. 
    As $\rank A_{11} = \rank A$, the row space of $A$ is spanned by the first $r$ rows, i.e., there exists $C \in K^{(t-r) \times r}$ such that
        \begin{equation*}
            A_{21} = C A_{11} 
            \andd 
            A_{22} = C A_{12}.
        \end{equation*}
        As $A_{11}$ is invertible, we may solve for $C$ using the first equation and substitute it in the second to obtain
    \begin{equation*} 
      A_{22} = A_{21} A_{11}^{-1} A_{12}.
    \end{equation*}
    Using this, we see that the $\ker A$ consists precisely of column vectors of the form
    \begin{equation*} 
      \begin{bmatrix}
        -A_{11}^{-1} A_{12} v \\
        v
      \end{bmatrix}_{t \times 1}
    \end{equation*}
    for $v \in K^{t-r}$. 
    In turn, if $B \in K^{t \times (t-1)}$ satisfies $A B = 0$, then every column of $B$ is of the above form, giving us
    \begin{equation*} 
      B = 
      \begin{bmatrix}
        -A_{11}^{-1} A_{12} W \\
        W
      \end{bmatrix}_{t \times (t-1)}
    \end{equation*}
    for some $W \in K^{(t-r) \times (t-1)}$. 
    Thus, we have the desired isomorphism
    \begin{align*} 
        \pi^{-1}(U) & \;\;\cong\;\; U \times K^{(t-r) \times (t-1)} \\
        (A, B) & \mapsto (A, \text{last $(t-r)$ rows of $B$}) \\
        \left(A, \smatrix{-A_{11}^{-1} A_{12} W \\ W} \right)& \mapsfrom (A, W),
    \end{align*}
    giving us the locally trivial fiber bundle. 

    It is clear that the fiber bundle~\eqref{eq:bundle} satisfies conditions (1) and (2) of Lemma~\ref{lem:cpdim-fiber}. 
    We now verify condition (3b): viewing $S$ as a subset of the affine space $K^{t \times t}$ corresponding to $K[X_{t \times t}]$, we note that $S$ is given by
    \begin{equation*} 
      S = \VV(I_{r+1}(X)) \setminus \VV(I_{r}(X)).
    \end{equation*}
    The smoothness of $S$ follows from the fact that
    the singular locus of $R \colonequals K[X]/I_{r+1}(X)$ is defined by $I_{r}(X)$, see \cite[Proposition~7.3.4]{BrunsHerzog}. 
    As $S$ is a nonempty open subset of $\VV(I_{r+1}(X))$, we obtain
    $\dim S = \dim R = (2t - r)r$, where the last equality is \cite[Theorem~7.3.1]{BrunsHerzog}. 
    We computed earlier that $\cp S = r^{2}$ and that $S$ has critical cohomology group of rank one. 
    Thus, it only remains to be shown that $S$ can be covered by $k$ affines, where 
    $k = \dim S - \cp S + 1 = (2t - 2r)r + 1$. 
    This now follows from Lemma~\ref{lem:ara-det-inside-det}.  
  \end{proof}

\section{The special linear group}\label{sec:SL}

  Consider the natural action of the special linear group $G \colonequals \SL_{t}(K)$ on the polynomial ring $S \colonequals K[X_{m \times t},\, Y_{t \times n}]$, where $t, m, n \ge 1$ are integers. 
  It is clear that elements of the subring $R \colonequals K[\Delta_{t}(X),\, \Delta_{t}(Y),\, X Y]$ are invariant under the action of $\SL_{t}(K)$. 
  By \cite[Theorem~3.3]{DeConciniProcesi76}, we have the equality $R = S^{\SL_{t}(K)}$ when $K$ is infinite, and thus $\frakU\colonequals I_{t}(X) + I_{t}(Y) + I_{1}(X Y)$ is the nullcone ideal for this action. 

  \begin{remark} \label{rem:ara-bound-dim}
    If $\underline{f} \subset R$ is any homogeneous system of parameters for $R$, 
    then we also have $\rad(\underline{f}S) = \rad \frakU$, 
    giving us $\ara \frakU \le \dim R$. 
  \end{remark}

  We shall compute the arithmetic rank of $\frakU$ and show that the above is an equality when $\chr(K) \neq 2$. 
  Note that the interesting case is when $t \le \max(m, n)$, else $\frakU = I_{1}(X Y)$ and the arithmetic rank is computed in \JPSW[Theorem~4.1]. 
  We first compute $\dim R$.

  \begin{theorem} \label{thm:dim-aux-ring}
    Let $K$ be a field, and $t, m, n \ge 1$ be integers with $t \le \max(m, n)$. 
    Consider the subring
    $R \colonequals K[\Delta_{t}(X), \Delta_{t}(Y), X Y] \subset 
    K[X_{m \times t}, Y_{t \times n}]$. 
    We have $\dim R = (m + n - t)t + 1$. 
  \end{theorem}
  \begin{proof} 
    Without loss of generality, we may assume $n \le m$, so that $\Delta_{t}(X)$ is nonempty. 
    We set $d \colonequals (m + n - t)t + 1$, the purported dimension. 
    As $R$ is a finitely generated $K$-algebra, we have 
    $\dim R = \trdeg_{K}(\Frac(R))$, 
    so it suffices to show that the field $Q \colonequals \Frac(R)$ can be generated over $K$ by $d$ elements and no fewer. 
    Note that we have the equality 
    \begin{equation*} 
      Q \;=\; K(\Delta_{t}(X), \Delta_{t}(Y), X Y)
      \;=\; K(\Delta_{t}(X), X Y)
    \end{equation*}
    in view of the equation
    \begin{equation*} 
      \minor{i_{1}, \ldots, i_{t}}{}_{Y}
      =
      \frac
      {\minor{1, \ldots, t}{i_{1}, \ldots, i_{t}}_{X Y}}
      {\minor{1, \ldots, t}{}_{X}}.
    \end{equation*}
    We now describe a generating set for $Q$. We define 
    \begin{align*} 
      B_{1} &\colonequals \{\minor{1, \ldots, t}{}_{X}\} \cup 
      \left\{
      \minor{1, \ldots, \hat{i}, \ldots, t ,j}{}_{X} 
      : 1 < i \le t < j \le m\right\}, \\
      B_{2} &\colonequals 
      \{[X Y]_{i 1} : 1 \le i \le m\}, \\
      B_{3} & \colonequals 
      \{[X Y]_{i j} : 1 \le i \le t,\, 1 < j \le n\}.
    \end{align*} 
    We set $B \colonequals B_{1} \sqcup B_{2} \sqcup B_{3}$ and note that 
    \begin{equation*} 
      \md{B} 
      = (1 + (t-1)(m-t)) + m + t(n - 1)
      = d.
    \end{equation*}
    We wish to show that $Q = K(B)$; it suffices to show that 
    $\Delta_{t}(X) \subset K(B)$ and 
    $\Delta_{1}(X Y) \subset K(B)$.
    The former is shown in the proof of \BMMP[Theorem A.2], 
    where they show the stronger statement that
    $\Delta_{t}(X) \subset K(B_{1} \sqcup B_{2})$. 
    Suppose now that $q$ is an entry of $X Y$ that is not already included in $B$, i.e., 
    $q = [X Y]_{i j}$ for some $i \ge t+1$ and $j \ge 2$. 
    Consider the square matrix
    \begin{equation*} 
      M \colonequals 
      \begin{bmatrix}
        [X Y]_{i j} & X_{i 1} & \cdots & X_{i t} \\
        [X Y]_{1 j} & X_{1 1} & \cdots & X_{1 t} \\
        \vdots & \vdots & \ddots & \vdots \\
        [X Y]_{t j} & X_{t 1} & \cdots & X_{t t} \\
      \end{bmatrix}_{(t + 1) \times (t + 1)}.
    \end{equation*}
    The first column of $M$ is evidently a $K[Y]$-linear combination of the others, and hence $\det M = 0$. 
    Expanding the determinant along the first column and rearranging the expressions gives
    \begin{equation*} 
      [X Y]_{i j} = 
      \frac{-1}{\minor{1, \ldots, t}{}_{X}}
      \sum_{k = 1}^{t} (-1)^{k} 
      \minor{i, 1, \ldots, \hat{k}, \ldots, t}{}_{X} 
      [X Y]_{k j}
      \in K(B). 
    \end{equation*}

    We have thus proven $\dim R \le d$. 
        The reverse inequality is proven in the appendix as Theorem~\ref{thm:dim-lower-bound} by constructing $d$ elements that are algebraically independent over $K$.
  \end{proof}

    The discussion around~\eqref{eq:local-cohomology-injection} now immediately yields the arithmetic rank of the nullcone ideal in characteristic zero. 
    Recall that the \emph{local cohomological dimension} $\lcd(I)$ of an ideal $I \subsetneq S$ is the largest integer $k$ such that $\HH^{k}_{I}(S) \neq 0$. 

    \begin{corollary}
        Let $K$ be a field of characteristic zero and 
        $\frakU \colonequals I_{t}(X) + I_{t}(Y) + I_{1}(X Y)$ the nullcone ideal for action of $\SL_{t}(K)$ 
        on the polynomial ring
        $S \colonequals K[X_{m \times t}, Y_{t \times n}]$, where $t \le \max(m, n)$. 
        We have
        \begin{equation*}
            \pushQED{\qed}
            \ara \frakU = 
            \lcd \frakU = 
            (m + n - t) t + 1.
            \qedhere
            \popQED
        \end{equation*}
    \end{corollary}

    We now show that the above local cohomology obstruction vanishes in positive characteristic. 

    \begin{theorem} \label{thm:small-lcd}
    Let $m$, $n$, $t$ be positive integers with $1 < t < n \le m$. 
    Let $K$ be an algebraically closed field of positive characteristic, and let $X$ and $Y$ be matrices of indeterminates of sizes $m \times t$ and $t \times n$, respectively. 
    Then, the local cohomological dimension of the ideal 
    $\frakU \colonequals I_{t}(X) + I_{t}(Y) + I_{1}(X Y)$ 
    in the polynomial ring $S \colonequals K[X, Y]$ satisfies
    \begin{equation*} 
      \lcd \frakU < (m + n - t)t + 1.
    \end{equation*}
  \end{theorem}
  \begin{proof} 
    We utilise the techniques introduced in the proof of \cite[Theorem~4.2]{HJPS}.
    For nonnegative integers $i$, $j$ with $i + j \le t$, we define
    \begin{equation*} 
      \frakp_{i, j} \colonequals I_{i+1}(X) + I_{j+1}(Y) + I_{1}(XY).
    \end{equation*}
    By \cite{HunekeVoC,ConciniStricklandVoC}, $\frakp_{i, j}$ is a prime ideal of height
    \begin{equation*} 
      \htt \frakp_{i, j} = (m - i)(t - i) + (n - j)(t - j) + i j
    \end{equation*}
    with $S/\frakp_{i, j}$ Cohen--Macaulay.
    In turn, $\htt \frakp_{i, j} = \lcd \frakp_{i, j}$ by \cite[Proposition~III.4.1]{PS74} as we are in positive characteristic. 
    The variety $\VV(\frakp_{i, j})$ consists of pairs of matrices $(A, B)$ satisfying $A B = 0$ and $\rank A \le i$ and $\rank B \le j$. 
    On the other hand, the variety $\VV(\frakU)$ consists of pairs $(A, B)$ satisfying $A B = 0$ and $\rank A < t$ and $\rank B < t$;
    moreover, $A B = 0$ implies the inequality $\rank A + \rank B \le t$. 
    An application of the Nullstellensatz thus gives us
    \begin{equation*} 
      \rad \frakU = \bigcap_{\substack{i + j = t \\ i, j < t}} \frakp_{i, j} = \bigcap_{i = 1}^{t - 1} \frakp_{i, t-i}. 
    \end{equation*}
    For $\ell$ with $1 \le \ell < t$, we define
    \begin{equation*} 
      \frakU_{\ell} \colonequals \bigcap_{i = 1}^{\ell} \frakp_{i, t-i}.
    \end{equation*}
    We induce on $\ell$ to show that 
    \begin{equation*} 
      \lcd \frakU_{\ell} < m t + n t - t^{2} + 1
    \end{equation*} 
    for all $1 \le \ell < t$; 
    this gives us the desired statement since $\rad \frakU = \rad \frakU_{t-1}$. 
    For $\ell = 1$, we need to verify
    \begin{equation*} 
      \lcd \frakp_{1, t-1} = \htt \frakp_{1, t-1} 
      = m t + n - m - t + 1 < m t + n t - t^{2} + 1,
    \end{equation*}
    which holds since $nt-t^2-n+m+t= t(n-t) + (m-n) + t>0$ by our assumption. 
    Assume now that $\ell \ge 1$.
    By \JPSW[Equation~(4.3.2)], we have 
    \begin{equation*} 
      \frakp_{\ell, t-\ell-1} =
      \frakU_{\ell} + \frakp_{\ell+1, t-\ell-1},
    \end{equation*}
    and thus the Mayer--Vietoris sequence for local cohomology
    \begin{equation*} 
      \to\; \HH_{\frakU_{\ell}}^{k}(S) \oplus \HH_{\frakp_{\ell+1, t-\ell-1}}^{k}(S)
      \;\to\; \HH_{\frakU_{\ell+1}}^{k}(S)
      \;\to\; \HH_{\frakp_{\ell, t-\ell-1}}^{k+1}(S) \to
    \end{equation*} 
    yields
    \begin{equation*} 
      \lcd \frakU_{\ell+1} \;\le\; 
      \max\{
        \lcd \frakU_{\ell}, \,
        \lcd \frakp_{\ell+1, t-\ell-1}, \,
        \lcd \frakp_{\ell, t-\ell-1} -1
      \}.
    \end{equation*}
    Thus, as before, it suffices to prove that
    \begin{equation*} 
      \htt \frakp_{\ell+1, t-\ell-1} < m t + n t - t^{2} + 1
      \andd
      \htt \frakp_{\ell, t-\ell-1}-1 < m t + n t - t^{2} + 1. 
    \end{equation*}
    Using the formula for the height as above, these reduce to
    \begin{equation*} 
      (n-t)(t-\ell-1)+(m-\ell-1)(\ell+1)+1>0
      \andd 
      (n-t)(t-\ell-1)+(m-\ell-1)\ell+1>0,
    \end{equation*}
    both of which are true as each parenthetical term is nonnegative by our hypotheses. 
  \end{proof} 

  \begin{theorem} \label{thm:ecd-computation}
    Let $m, n, t$ be positive integers with $n \le m$ and $t \le m$. 
    Let $K$ be an algebraically closed field of characteristic other than two, and $X$ and $Y$ be matrices of indeterminates of sizes $m \times t$ and $t \times n$, respectively. 
    Consider the ideal
    $\frakU \colonequals I_{t}(X) + I_{t}(Y) + I_{1}(X Y)$
        in the polynomial ring $S \colonequals K[X, Y]$
    and the open complement
    \begin{equation*} 
      U \colonequals 
      \left(K^{m \times t} \times K^{t \times n}\right)
      \setminus
      \VV(\frakU).
    \end{equation*}
    If $t < m$ or $t=m=n+1$, then
    \begin{equation} \label{eq:ecd-statement}
      \ecd(U) = 2mt + 2nt - t^{2}.
    \end{equation}
  \end{theorem}
  \begin{proof} 
    For brevity, we set $\AA \colonequals K^{m \times t} \times K^{t \times n}$ to be the affine space corresponding to the polynomial ring $S$. 
    In view of \Poincare\ duality~\eqref{eq:PD}, it suffices to prove that
    \begin{equation} \label{eq:cpd-statement}
      \cp(U) = t^{2}. 
    \end{equation}

    \textbf{Case 1.} We have $t \le n \le m$. 

    Our variety of interest $W \colonequals \VV(\frakU)$ is a closed subset of $V \colonequals \VV(X Y)$, 
    and the complement is the open set $U' \colonequals V \setminus W$ consisting of pairs $(A, B) \in \AA$ satisfying $A B = 0$ and that one of $A$ or $B$ is of rank $t$; 
    however, this then implies that the other is zero. 
    In other words, 
    \begin{align*} 
      U' = 
      (\GL(m, t) \times \{0\})
      \sqcup
      (\{0\} \times \GL(n, t)^{\tr}),
    \end{align*}
    where $\GL(k, t) = \{M \in K^{k \times t} : \rank M = t\}$. 
    We let $U_{1}$ and $U_{2}$ denote the respective components above. 
    By \JPSW[Theorem~8.4 and Lemma~8.2], we have
    \begin{equation*} 
      \cp(V) = \cp(U_{1}) = \cp(U_{2}) = t^{2}
      \andd
      \Hc^{t^{2}}(V) \cong \Hc^{t^{2}}(U_{1}) \cong \Hc^{t^{2}}(U_{2}) \cong \ZZ/2.
    \end{equation*}
    As the $U_{i}$ are open in $U$, we have $\Hc^{\ast}(U) \cong \Hc^{\ast}(U_{1}) \oplus \Hc^{\ast}(U_{2})$,
    giving us $\cp(U') = t^{2}$ and $\Hc^{t^{2}}(U') \cong (\ZZ/2)^{2}$. 
    Analysing the long exact sequence~\eqref{eq:les-triple} for the decomposition $V = W \sqcup U'$,
    we see that
    \begin{equation*} 
      \Hc^{i}(W)
      =
      \begin{cases}
        0 & \text{ if } i < t^{2} - 1, \\
        \ker\left((\ZZ/2)^{2} \to \ZZ/2\right) & \text{ if } i = t^{2}-1.
      \end{cases}
    \end{equation*}
    The displayed kernel must be nonzero, giving us $\cp(W) = t^{2} - 1$. 
    Our hypothesis gives us $\cp \AA = 2(m t + n t) \ge t^{2}$. 
    Analysing~\eqref{eq:les-triple} for the decomposition $\AA = U \sqcup W$ gives us $\cp U = t^{2}$, as desired~\eqref{eq:cpd-statement}. 

    \textbf{Case 2.} We have $n < t < m$. 

    As $\frakU = I_{t}(X) + I_{1}(X Y)$, we see that $U = U_{1} \cup U_{2}$, where $U_{1}$ and $U_{2}$ are the respective complements (in $\AA$) of $\VV(I_{t}(X))$ and $\VV(I_{1}(X Y))$. 
    As topological spaces, we have 
    \begin{equation*} 
      U_{1} = [K^{m \times t} \setminus \VV(I_{t}(X))] \times K^{t \times n},
    \end{equation*}
    giving us $\ecd(U_{1}) = \ecd(K^{m \times t} \setminus \VV(I_{t}(X)))$. 
    By \cite[Lemma~2']{BrunsSchwanzl}, we then obtain
    \begin{equation*} 
      \ecd(U_{1}) = 2 m t - t^{2}.
    \end{equation*} 
    By \JPSW[Theorem~4.1], $\ara(I_{1}(X Y)) = m n$, 
    and thus~\eqref{eq:affine-vanishing} yields
    \begin{equation*} 
      \ecd(U_{2}) \le m n + m t + n t - 1.
    \end{equation*} 
    Note that if $(A, B) \in U_{1} \cap U_{2}$, then $A$ is full rank and $A B \neq 0$.
    As $A$ is full rank, $A B \neq 0$ is equivalent to $B \neq 0$. 
    Thus,
    \begin{align*} 
      U_{1} \cap U_{2} &= 
      \{(A, B) \in \AA : I_{t}(A) \neq 0 \text{ and } A B \neq 0\} \\
      &= 
      \{(A, B) \in \AA : I_{t}(A) \neq 0 \text{ and } B \neq 0\} \\
      &=
      [K^{m \times t} \setminus \VV(I_{t}(X))] \times [K^{t \times n} \setminus \{0\}],
    \end{align*}
    giving us 
    \begin{equation*} 
      \ecd(U_{1} \cap U_{2})
      = \ecd(K^{m \times t} \setminus \VV(I_{t}(X))) + \ecd(K^{t \times n} \setminus \{0\})
      = (2 m t - t^{2}) + (2 n t - 1), 
    \end{equation*}
    where the value of $\ecd(K^{m \times t} \setminus \VV(I_{t}(X)))$ may be found in \cite[Lemma~2']{BrunsSchwanzl}. 
    Comparing the three displayed \etale\ cohomological dimensions, we see that $\ecd(U_{1} \cap U_{2})$ is strictly the largest;
    $\ecd(U_{1} \cap U_{2}) > \ecd(U_{1})$ is clear, 
    and
    \begin{equation*} 
      \ecd(U_{1} \cap U_{2}) - \ecd(U_{2}) \ge (m - t)(t - n) > 0. 
    \end{equation*}
    A straightforward application of Mayer--Vietoris~\eqref{eq:les-mv} then gives the desired equality~\eqref{eq:ecd-statement}
    \begin{equation*} 
      \ecd(U) = \ecd(U_{1} \cap U_{2}) + 1 = 2 m t + 2 n t - t^{2}.
    \end{equation*}

    \textbf{Case 3.} We have $n < t = m$. 

    Our hypothesis then implies $n=t-1$. Set $V \colonequals \VV(\frakU) \subseteq \AA$ and note that we have a decomposition
    $V = \bigsqcup_{r = 0}^{t - 1} V_{r}$, where
    \begin{equation*} 
      V_{r} \colonequals 
      \{
        (A, B) \in K^{t \times t} \times K^{t \times (t - 1)}
        :
        \rank A = r \andd A B = 0
      \}.
    \end{equation*}
    Then, by Proposition~\ref{prop:cpd-Vr}, we have 
    \begin{equation*}
            \cp V_{r} = r^{2} + 2(t - r)(t - 1) = r^{2} - 2(t - 1) r + 2 t (t - 1),
        \end{equation*}
    giving us $\cp V_{0} > \cdots > \cp V_{t-1}$. 
    Moreover, each $V_{j}$ is open in the union $\bigcup_{r=0}^{j} V_{r}$, being the complement of $\VV(I_j(X))$ in the union. 
    We thus obtain $\cp V = \cp V_{t-1}$ by Lemma~\ref{lem:cpd-locally-closed-union}, 
    giving us $\cp V = t^{2} - 1$. 
    On the other hand, $\cp \AA = 2(m t + n t) \ge t^{2}$, so the long exact sequence~\eqref{eq:les-triple} applied to $\AA = V \sqcup U$ yields $\cp U = t^{2}$ as desired~\eqref{eq:cpd-statement}. 
  \end{proof}

  \begin{theorem} \label{thm:ara-SL}
    Let $m, n, t$ be positive integers with $t \le \max(m, n)$. 
    Let $K$ be a field of characteristic other than two, and let $X$ and $Y$ be matrices of indeterminates of sizes $m \times t$ and $t \times n$, respectively. 
    Then, the arithmetic rank of the ideal 
    $\frakU \colonequals I_{t}(X) + I_{t}(Y) + I_{1}(X Y)$ 
    in the polynomial ring $S \colonequals K[X, Y]$ is
    $(m + n - t)t + 1$.
  \end{theorem}
  \begin{proof} 
    Let $d \colonequals (m + n - t)t + 1$, 
    the purported arithmetic rank.  
    In view of Remark~\ref{rem:ara-bound-dim} and Theorem~\ref{thm:dim-aux-ring},
    it suffices to show that $\ara \frakU \ge d$. 
    As the arithmetic rank can only decrease when passing to a field extension, we may assume that $K$ is algebraically closed. 
    Without loss of generality, we may assume that $n \le m$. 
    Setting $U \colonequals (K^{m \times t} \times K^{t \times n}) \setminus \VV(\frakU)$, 
    by~\eqref{eq:affine-vanishing}, we have the desired lower bound whenever we have $\ecd(U) \ge 2 m t + 2 n t - t^{2}$. 
    By Theorem~\ref{thm:ecd-computation}, we have this in all cases except $n+1 < t = m$. 
    Assuming now that we are in this case. 
    We wish to prove that the ideal
    $\frakU = (\det X) + I_{1}(X Y)$ has arithmetic rank $d = m n + 1$. 
    Note that increasing $n$ by one increases both $d$ and the number of minimal generators of $\frakU$ by~$m$. 
    In view of Lemma~\ref{lem:one-more-generator}, we may reduce to the earlier case $n=t-1$ and we are done.
  \end{proof}

    \begin{remark}
        In Corollary~\ref{cor:STgens}, for $\min(m,n)<t\le\max(m,n)$, we construct an explicit minimal set-theoretic generating set of the ideal $I_t(X)+I_t(Y)+I_1(XY)$.
        The construction works over any base ring and in particular, we obtain the arithmetic rank of the corresponding ideal in the ring $\ZZ[X, Y]$.
    \end{remark}

\section{The orthogonal and symplectic groups}
\label{sec:OSp}

  We now note that for the orthogonal and symplectic groups, the case with the additional copies of the dual representation reduces to the case of just copies of the regular representation.
    In what follows, we use the convention that $\binom{n}{2}=0$ if $n < 0$.

  \subsection{The orthogonal group}
  
    Consider the action of the orthogonal group $G \colonequals \OO_{t}(K)$ on the polynomial ring $S \colonequals K[X_{m \times t}, Y_{t \times n}]$, where $t, m, n \ge 1$ are integers. 
    We consider the new polynomial ring $R \colonequals K[Z_{t \times (m + n)}]$ with the action of $M \in G$ given by 
    $M \colon Z \mapsto M Z$. 
    As $M^{-1} = M^{\tr}$ for $M \in G$, we see that there is a degree-preserving $G$-equivariant $K$-algebra isomorphism $\varphi \colon R \to S$ given by
    $Z \mapsto \begin{bmatrix} X^{\tr} & Y \end{bmatrix}$;
        recall that $G$-equivariance means that
        $\varphi(M(f)) = M(\varphi(f))$ for all $M \in G$ and $f \in R$. 
        Note that under this map, we have
        \begin{equation*}
            Z^{\tr} Z
            \mapsto
            \begin{bmatrix}
                X X^{\tr} & X Y \\
                (X Y)^{\tr} & Y^{\tr} Y
            \end{bmatrix}.
        \end{equation*}
    This gives us a commutative diagram
    \begin{equation*} 
      \begin{tikzcd}
      {K[Z]} \arrow[r, "\cong"] & {K[X, Y]} \\
      {K[Z^{\tr} Z]} \arrow[r, "\cong"] \arrow[u, hook] & {K[X X^{\tr}, Y^{\tr} Y, X Y]}  \arrow[u, hook]
      \end{tikzcd}
    \end{equation*}
    of graded $K G$-modules. 
    When $K$ is infinite and of characteristic other than two, 
    it is known \cite[Theorem~5.6]{DeConciniProcesi76} that $K[Z]^{G} = K[Z^{\tr} Z]$, and as $\varphi$ is $G$-equivariant, we have the equality $K[X, Y]^{G} = K[X X^{\tr}, Y^{\tr} Y, X Y]$. 

    The arithmetic rank of the nullcone ideal now follows from \JPSW[Theorem~5.1]:

    \begin{theorem}
        \label{thm:ara-O}
      Let $K$ be a field of characteristic other than two, and let $X$ and $Y$ be matrices of indeterminates of sizes $m \times t$ and $t \times n$, respectively. 
      Then, the arithmetic rank of the ideal 
      $I_{1}(X X^{\tr}) + I_{1}(Y^{\tr} Y) + I_{1}(X Y)$
      in $K[X, Y]$ is
      \begin{equation*} 
        \pushQED{\qed} 
        \binom{n+m+1}{2} - \binom{n+m+1-t}{2}.
        \qedhere \popQED 
      \end{equation*}
    \end{theorem}

  \subsection{The symplectic group}
    
    Consider the action of the symplectic group $G \colonequals \Sp_{2t}(K)$ on the polynomial ring $S \colonequals K[X_{m \times 2t}, Y_{2t \times n}]$, where $t, m, n \ge 1$ are integers. 
    We consider the new polynomial ring $R \colonequals K[Z_{2t \times (m + n)}]$ with the action of $M \in G$ given by 
    $M \colon Z \mapsto M Z$. 
    As $M^{-1}=\Omega^{-1}M^{\tr}\Omega$, we have a degree-preserving $G$-equivariant $K$-algebra isomorphism $\varphi \colon R \to S$ given by
    $Z \mapsto \begin{bmatrix} \Omega X^{\tr} & Y \end{bmatrix}$, 
    giving us the diagram
    \begin{equation*} 
      \begin{tikzcd}
           {K[Z]} \arrow[r, "\cong"] & {K[X, Y]} \\
           {K[Z^{\tr} \Omega Z]} \arrow[r, "\cong"] \arrow[u, hook] & {K[X \Omega X^{\tr}, Y^{\tr} \Omega Y, X Y]} \arrow[u, hook]
      \end{tikzcd},
    \end{equation*}
    where the lower rings are the invariant subrings when $K$ is infinite; for the left inclusion, this is due to \cite[Theorem~6.6]{DeConciniProcesi76}, and the right follows from $G$-equivariance as before. 

    The arithmetic rank of the nullcone ideal now follows from \JPSW[Theorem~3.1]:

    \begin{theorem}
        \label{thm:ara-Sp}
      Let $K$ be a field of characteristic other than two, and let $X$ and $Y$ be matrices of indeterminates of sizes $m \times 2t$ and $2t \times n$, respectively. 
      Then, the arithmetic rank of the ideal 
      $I_{1}(X \Omega X^{\tr}) + I_{1}(Y^{\tr} \Omega Y) + I_{1}(X Y)$
      in $K[X, Y]$ is
      \begin{equation*} 
        \pushQED{\qed} 
        \binom{n+m}{2} - \binom{n+m-2t}{2}.
        \qedhere \popQED 
      \end{equation*}
    \end{theorem}

    \begin{remark}
        We refer the reader to \JPSW[\S3 and \S5] to note that the above arithmetic ranks agree with the dimension of the corresponding invariant ring, and that the local cohomology obstructions vanish in positive characteristic. 
    \end{remark}

\appendix \section{A lower bound on the dimension of the invariant ring} \label{sec:Appendix}

  Given integers $m$, $n$, $t$ such that $t\le\max(m,n)$, set $S\colonequals K[X_{m \times t},\, Y_{t \times n}]$, recall the invariant subring $R=K[\Delta_{t}(X),\, \Delta_{t}(Y),\, X Y]$ of the $\SL_t(K)$-action on $S$ from Section~\ref{sec:SL}. 
  In this appendix, we produce a proof of a lower bound on the dimension of $R$. 
  An analogous result was proven in \cite{LakshmibaiShukla} for $t\le\min(m,n)$.  

  Let $q \colonequals \min(m,n,t)$. 
  Note that for $r>q$, all $r$-minors of $XY$ vanish. 
  Define 
  \begin{equation*} 
    B:=\Delta_t(X) \cup \Delta_t(Y) \cup \bigcup_{i=1}^{t}\Delta_i(XY).
  \end{equation*} 
  Let $1\le a_1<\cdots<a_t$ and $1\le b_1<\cdots<b_t$ be sequences of integers.
  We use the notations 
  $\minor{a_1,\ldots,a_t}{}_{X}$, 
  $\minor{b_1,\ldots,b_t}{}_{Y}$, and 
  $\minor{a_1,\ldots,a_i}{b_1,\ldots,b_i}$ 
  to denote the corresponding minors of $X$, $Y$, and $XY$ respectively.

  We introduce a partial order on $B$ as the transitive closure of the following relations: 

  \begin{itemize}
    \item $\minor{a_1,\ldots,a_t}{}_{X} \le \minor{b_1,\ldots,b_t}{}_{X}$ if $a_i\le b_i$ for all $i\in [t]$.
    \item $\minor{a_1,\ldots,a_t}{}_{Y} \le \minor{b_1,\ldots,b_t}{}_{Y}$ if $a_i\le b_i$ for all $i\in [t]$.
    \item $\minor{a_1,\ldots,a_r}{b_1,\ldots,b_r} \le \minor{c_1,\ldots,c_s}{d_1,\ldots,d_s}$ if $r \ge s$ and $a_i\le c_i$ and $b_i\le d_i$ for all $i\in[s]$. 
    \item $\minor{c_1,\ldots,c_t}{}_{X} \le \minor{a_1,\ldots,a_r}{b_1,\ldots,b_r}$ if $c_i\le a_i$ for all $i \in [r]$.
    \item $\minor{d_1,\ldots,d_t}{}_{Y} \le \minor{a_1,\ldots,a_r}{b_1,\ldots,b_r}$ if $d_i\le b_i$ for all $i \in [r]$.
  \end{itemize}

  The elements of $\Delta_t(X)$ and $\Delta_t(Y)$ are incomparable and form two branches at the
  bottom of the poset, while the minors of $XY$ lie above them. 
  In particular, a chain cannot contain both a maximal minor of $X$ and a maximal minor of $Y$. 
  We call a product of elements of $B$ a \emph{standard monomial} if its factors, with repetition allowed, form a chain in $B$.

  \begin{example}\label{eg:poset}
    Suppose that $m=3$, $t=2$, and $n=2$.  The Hasse diagram of $B$ is:
    
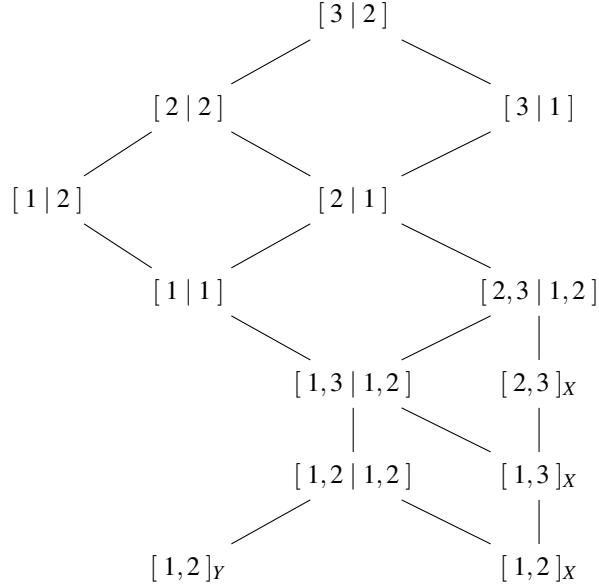
\begin{equation*} \label{eq:poset-example}
      \begin{tikzcd}[row sep=1.7em, column sep=1.6em]
      && {\minor{3}{2}} \\
      & {\minor{2}{2}} && {\minor{3}{1}} \\
      {\minor{1}{2}} && {\minor{2}{1}} \\
      & {\minor{1}{1}} && {\minor{2,3}{1,2}} \\
      && {\minor{1,3}{1,2}} & {\minor{2,3}{}_X} \\
      && {\minor{1,2}{1,2}} & {\minor{1,3}{}_X} \\
      & {\minor{1,2}{}_Y} && {\minor{1,2}{}_X}
      \arrow[no head, from=7-2, to=6-3]
      \arrow[no head, from=7-4, to=6-3]
      \arrow[no head, from=7-4, to=6-4]
      \arrow[no head, from=6-3, to=5-3]
      \arrow[no head, from=6-4, to=5-3]
      \arrow[no head, from=6-4, to=5-4]
      \arrow[no head, from=5-3, to=4-2]
      \arrow[no head, from=5-3, to=4-4]
      \arrow[no head, from=5-4, to=4-4]
      \arrow[no head, from=4-2, to=3-1]
      \arrow[no head, from=4-2, to=3-3]
      \arrow[no head, from=4-4, to=3-3]
      \arrow[no head, from=3-1, to=2-2]
      \arrow[no head, from=3-3, to=2-2]
      \arrow[no head, from=3-3, to=2-4]
      \arrow[no head, from=2-2, to=1-3]
      \arrow[no head, from=2-4, to=1-3]
      \end{tikzcd}
    \end{equation*}
    Any maximal chain has seven elements, in agreement with
    $mt+nt-t^2+1=7$.
  \end{example}

  Observe that $B$ is a ranked poset because it is composed of ranked posets on each of $\Delta_t(X)$, $\Delta_t(Y)$, and $\Delta_r(XY)$ for $r \in [t]$. We record the rank of the poset $B$ for later use. 
  As before, the \emph{rank} of a poset refers to the longest length of a chain.

  \begin{lemma}\label{lem:poset_rank}
    The rank of $B$ is $mt+nt-t^2+1$.
  \end{lemma}
  \begin{proof}
    Without loss of generality, we may assume $m\ge t$, so that $\Delta_t(X)$ is nonempty. Let $q \colonequals \min(n,t)$ as before.
    We construct a maximal chain in $B$, starting with the largest element of $B$. 
    As $B$ is a ranked poset, the length of such a maximal chain will determine the rank. 
    
    We begin with the $n+m-1$ elements 
    \begin{equation*} 
      \minor{m}{n},\minor{m}{n-1},\ldots,\minor{m}{1},\, \minor{m-1}{1},\ldots,\minor{1}{1}.
    \end{equation*}
    Moving on to the $2$-minors of $XY$, the next element in the chain is $\minor{1,m}{1,n}$. 
    We proceed from $\minor{1,m}{1,n}$ to $\minor{1,2}{1,2}$ as before: by first decreasing the second $X$-index to $2$ while keeping the $Y$-indices constant, and then vice versa. 
    This gives us a chain segment of length $m+n-3$. 
    Proceeding similarly for $3 \le r \le q$, we obtain chain segments of length $m+n-(2r-1)$ for each $r$. 
    The smallest element in our chain so far is $\minor{1,\ldots,q}{1\ldots,q}$.

    \textbf{Case 1.} We have $q = t \le n$.

    In this case, the next and final element in our chain must be $\minor{1,\ldots,q}{}_X$. The total length of the chain is thus 
    \begin{equation*}
        1+\sum_{r=1}^t (m+n-2r+1)=mt+nt-t^2+1.
    \end{equation*}

    \textbf{Case 2.} We have $q = n < t$.

    We can continue the chain with the element $\minor{1,\ldots,n,m-(t-n)+1,\ldots,m}{}_X$. 
    Descending through the poset structure to $\minor{1,\ldots,t}{}_X$ decreases the sum of the indices by exactly one at each step,
    giving us a chain segment of length $(t-n)(m-t)+1$. 
    The total length of the chain is thus
    \begin{equation*}
        (t-n)(m-t)+1+\sum_{r=1}^n (m+n-2r+1)=mt+nt-t^2+1. \qedhere
    \end{equation*}
  \end{proof}
    
  We next associate a pair of tableaux $(T_{X}, T_{Y})$ to a standard monomial. 
  Write the factors of a given standard monomial
  in weakly increasing order in $B$. 
  The factor $\minor{i_1,\ldots,i_t}{}_{X}$ contributes
  the row $(i_1,\ldots,i_t)$ to the $X$-tableau $T_{X}$, 
  the factor  $\minor{j_1,\ldots,j_t}{}_{Y}$ contributes 
  $(j_1,\ldots,j_t)$ to the $Y$-tableau $T_{Y}$, 
  and the factor $\minor{a_1,\ldots,a_r}{b_1,\ldots,b_r}$ contributes
  $(a_1,\ldots,a_r)$ to $T_{X}$ and 
  $(b_1,\ldots,b_r)$ to $T_{Y}$. 
  Empty rows are omitted. 

  For example, the standard monomial
  \begin{equation} \label{eq:tableaux-example}
    \minor{1,2}{}_{X} \, \minor{1,3}{}_{X} \, \minor{1,3}{1,2} \, \minor{1}{1} \, \minor{2}{1}^2
  \end{equation}
  corresponds to the pair $(T_{X}, T_{Y})$ displayed below, where we draw the tableaux by aligning their bottom rows.  
  \begin{equation} \label{eq:tableaux-correspondence}
    \begin{minipage}[b][][]{0.4\textwidth}
      \flushright
      \begin{ytableau}
        1 & 2 \\
        1 & 3 \\
        1 & 3 \\
        1 \\
        2 \\
        2
      \end{ytableau} 
    \end{minipage}
    \quad
    \begin{minipage}[b][][]{0.4\textwidth}
      \flushleft
      \begin{ytableau}
      1 & 2 \\
      1 \\
      1 \\
      1
      \end{ytableau}
    \end{minipage}%
  \end{equation}

  The definition of the order shows that the row lengths weakly decrease, the entries in each row
  are strictly increasing, and the entries in each column are weakly increasing.
  Thus, a standard monomial gives a pair $(T_X,T_Y)$ of semistandard Young
  tableaux. 
  Conversely, such a pair recovers the factors of the standard monomial by proceeding upward in the diagram.

  We begin by showing that the standard monomials in $B$ are linearly independent over $K$. 
  To do so, equip $S$ with the lexicographic monomial order induced by
  \begin{equation*} \label{eq:monomial-order}
   x_{11}>x_{12}>\cdots>x_{1t}>x_{21}>\cdots>x_{mt}
   >y_{11}>y_{12}>\cdots>y_{1n}>y_{21}>\cdots>y_{tn}.
  \end{equation*}
  For $a_1<\cdots<a_t$ and $b_1<\cdots<b_t$, the initial monomials are given as
  \begin{align}
    \init(\minor{a_1,\ldots,a_t}{}_{X})
    &=\prod_{i=1}^t x_{a_{i} i}, \label{eq:initial-X}\\
    \init(\minor{b_1,\ldots,b_t}{}_{Y})
    &=\prod_{i=1}^t y_{ib_i}, \label{eq:initial-Y}\\
    \init(\minor{a_1,\ldots,a_r}{b_1,\ldots,b_r})
    &=\prod_{i=1}^r x_{a_i i}y_{ib_i}. \label{eq:initial-XY}
  \end{align}
  The first two equalities follow by expanding the relevant determinants.  For the  third, apply the Cauchy--Binet formula to the indicated minor of $XY$. 
  The largest summand is the one using columns $1,\ldots,r$ of $X$ and rows
  $1,\ldots,r$ of $Y$, and the diagonal term is largest in each of the two
  resulting minors. This gives~\eqref{eq:initial-XY}.

  Let us illustrate the reconstruction that will be used in the proof below.
  For the standard monomial in~\eqref{eq:tableaux-example},
  equations~\eqref{eq:initial-X}--\eqref{eq:initial-XY} give
  \begin{equation*}
    \init\big(\minor{1,2}{}_{X} \, \minor{1,3}{}_{X} \, \minor{1,3}{1,2} \, \minor{1}{1} \, \minor{2}{1}^2\big)
    \;=\;
    x_{11}^{4} x_{21}^{2} x_{22}^{\ph} x_{32}^{2} y_{11}^{4} y_{22}^{\ph}.
  \end{equation*}
  The variables of the form $x_{i1}$ give the first column of the 
  $X$-tableau: 
  the factor $x_{11}^{4} x_{21}^{2}$ gives the multiset $\{1,1,1,1,2,2\}$. 
  Similarly, $x_{22}^{\ph} x_{32}^{2}$ gives the multiset $\{2,3,3\}$ in its second column. 
  Arranging these multisets in weakly increasing order recovers $T_{X}$ from~\eqref{eq:tableaux-correspondence}. 
  On the other hand, the variables of the form $y_{ij}$ give the entries in the $i$-th column of the $Y$-tableau: $y_{11}^{4}$ gives the multiset $\{1,1,1,1\}$ in the first column and $y_{22}$ gives the single entry $2$ in the second column,
  recovering $T_{Y}$ from~\eqref{eq:tableaux-correspondence}. 
  We now show that this procedure may be used in general to recover the pair of tableaux from the initial monomial of a standard monomial. 

  \begin{lemma} \label{lem:standard-independent}
  Distinct standard monomials have distinct initial monomials. 
  Consequently, the standard monomials are $K$-linearly independent in $S$.
  \end{lemma}
  \begin{proof}
   It suffices to show that the initial monomial of a standard monomial determines its pair of tableaux uniquely. 
   Let $M$ be a standard monomial with associated tableaux $(T_X, T_Y)$, 
   and define 
   \begin{equation*} 
    d^X_{ij} \colonequals \deg_{x_{ij}}(\init(M)) 
    \andd 
    d^Y_{jk} \colonequals \deg_{y_{jk}}(\init(M)).
   \end{equation*}
   We make the following observations:
   \begin{enumerate}
    \item The respective number of boxes in the $j$-th column of $T_X$ and $T_{Y}$ are given by
    \begin{equation*} 
        \sum_{i=1}^{m} d^X_{ij}
        \andd
        \sum_{k=1}^{n} d^Y_{jk}.
    \end{equation*}
      \item The respective multisets of entries in the $j$-th column of $T_{X}$ and $T_{Y}$ are given by
      \begin{equation*} 
        \bigcup_{i=1}^{m}\{i\}^{d^X_{ij}} 
        \andd
        \bigcup_{k=1}^{n}\{k\}^{d^Y_{jk}},
      \end{equation*}
    where $\{i\}^{d}$ denotes the multiset consisting of $d$ copies of
    $i$.
   \end{enumerate}
   Every multiset has a unique weakly increasing arrangement. 
   As the columns of $T_X$ and $T_Y$ are weakly increasing, the preceding data recovers $T_X$ and $T_Y$ uniquely, and, in turn, determines $M$. 
   Distinct standard monomials have distinct initial monomials, and are therefore $K$-linearly independent.
  \end{proof}

  We may now obtain the desired dimension bound.

  \begin{theorem} \label{thm:dim-lower-bound}
    Let $K$ be any field and $m$, $n$, $t$ positive integers with $t\le\max(m,n)$. For
    \begin{equation*}
      R=K[\Delta_t(X),\Delta_t(Y),XY]
      \subseteq K[X_{m\times t},Y_{t\times n}],
    \end{equation*}
    one has
    \begin{equation*}
      \dim R \ge mt+nt-t^2+1.
    \end{equation*}
  \end{theorem}
  \begin{proof}
    By Lemma~\ref{lem:poset_rank}, there exists a chain $f_{1}, \ldots, f_{d} \in B$ with $d=mt+nt-t^2+1$ elements. 
    By construction, elements of the form
    \begin{equation*}
        f_1^{e_1} \cdots f_d^{e_d},
        \qquad e_1,\ldots,e_d\ge0,
    \end{equation*}
    are standard monomials, and thus, $K$-linearly independent by
    Lemma~\ref{lem:standard-independent}; 
    this is precisely the statement that $f_{1}, \ldots, f_{d}$ are algebraically independent.
    Thus, we obtain the desired bound $\dim R = \trdeg_{K} \Frac(R) \ge d$. 
  \end{proof}

  \begin{remark}
    In light of Theorem~\ref{thm:dim-aux-ring}, we obtain that $\dim(R)=\rank(B)$. This might lead one to expect that $R$ is an ASL on $(B,<)$.
        
        For $t>\min(m,n)$, the ring $R$ is indeed an ASL: the straightening relations are stated in \cite[Theorem~5.1.1]{LakshmibaiShukla}. 
    However, for $t \le \min(m,n)$, this may no longer be true. 
    This can be readily observed in Example~\ref{eg:poset}: 
    the elements $\minor{1,2}{}_{Y}$ and $\minor{1,2}{}_{X}$ are incomparable but the product $\minor{1,2}{}_{Y} \minor{1,2}{}_{X}$ cannot be straightened as $B$ has no element smaller than $\minor{1,2}{}_{Y}$ and $\minor{1,2}{}_{X}$.
    Reversing the poset order does not yield an ASL either, as 
    \begin{equation*} 
      \minor{1}{2}\minor{2}{1} \;=\; \minor{1}{1}\minor{2}{2}-\minor{1,2}{1,2}.
      \qedhere
    \end{equation*}
  \end{remark}

    The \emph{rank} of an element $\mu$ in a poset $H$ is the length of the longest chain with maximal element $\mu$. When $R$ is an ASL, the following proposition allows us to construct a minimal set-theoretic generating set of the nullcone ideal of the $\SL_t$-action considered in Section~\ref{sec:SL}.

    \begin{prop}[{\cite[Lemma~5.9]{BrunsVetter}}]\label{prop:SOPofASL}
        Let $H$ be a finite poset and $A=K[H]$ be an ASL\@. Let $x_i=\sum_{\text{rank}(\mu)=i}\mu$ for all $1\le i\le \rank(H)$. Then, $\rad (x_1,\dots,x_{rank(H)}) = \rad (H)$. \qed
    \end{prop}

    \begin{corollary}\label{cor:STgens}
        Let $K$ be a field of characteristic other than two and $m$, $n$, $t$ be positive integers with $\min(m,n)<t\le\max(m,n)$. Then,
        \[
        \Lambda = \left\{ \sum_{\substack{\operatorname{rank}(\mu)=i}} \mu \;:\; i=1,\ldots,mt+nt-t^2+1 \right\}
        \]
        is a minimal set-theoretic generating set of the ideal $\frakU=I_t(X)+I_t(Y)+I_1(XY)$ in the polynomial ring $K[X_{m \times t},Y_{t \times n}]$. 
    \end{corollary}
    \begin{proof}
        Let $R\colonequals K[\Delta_t(X),\Delta_t(Y),XY]$. By Proposition~\ref{prop:SOPofASL}, $\Lambda$ generates the ideal $\frakU \cap R$ up to radical in $R$, and hence, the ideal $\frakU$ up to radical in $K[X,Y]$. As $\ara(\frakU)=mt+nt-t^2+1$ by Theorem~\ref{thm:ara-SL}, we conclude that $\Lambda$ is a minimal set-theoretic generating set of $\frakU$.
    \end{proof}

    \begin{remark}
        In the setting of Corollary~\ref{cor:STgens}, the ASL structure also enables us to construct a homogeneous minimal set-theoretic generating set of the ideal $\frakU$. For $i=1,\ldots,\ara(\frakU)$, let $d_i$ be the least common multiple of $\deg(\mu)$ for all elements $\mu$ of rank $i$, and for such a $\mu$, set $\alpha(\mu) \colonequals d_i/\deg(\mu)$. Then, the set
        \[
        \left\{ \sum_{\substack{\operatorname{rank}(\mu)=i}} \mu^{\alpha(\mu)} \;:\; i=1,\ldots,mt+nt-t^2+1 \right\}
        \] is homogeneous and generates the ideal $\frakU$ up to radical.
    \end{remark}

\printbibliography

\end{document}